\documentclass[11pt, pdf]{amsart}
\usepackage{latexsym}
\usepackage{amssymb,amsmath,mathabx, amscd}
\usepackage{enumerate}
\usepackage[margin=1in]{geometry}
\usepackage{tikz-cd}
\usepackage{MnSymbol}
\usepackage{hyperref}
\usepackage{mathrsfs, colonequals}

\newcommand{\cc}{\mathbb{C}}

\newcommand{\pp}{\mathbb{P}}

\newcommand{\I}{\mathcal{I}}

\renewcommand{\O}{\mathcal{O}}

\newcommand{\Sym}{\operatorname{Sym}}
\newcommand{\Pic}{\operatorname{Pic}}

\newcommand{\Hilb}{\operatorname{Hilb}}

\newcommand{\Bl}{\operatorname{Bl}}

\newcommand{\Hom}{\operatorname{Hom}}
\newcommand{\Ram}{\operatorname{Ram}}
\newcommand{\rk}{\operatorname{rk}}
\newcommand{\ev}{\operatorname{ev}}

\renewcommand{\bar}{\overline}

\newtheorem*{thm*}{Theorem}
\newtheorem{thm}{Theorem}[section]

\newtheorem{lem}[thm]{Lemma}

\newtheorem{prop}[thm]{Proposition}
\newtheorem*{prop*}{Proposition}
\newtheorem*{question}{Main Question}

\theoremstyle{definition}
\newtheorem{defin}[thm]{Definition}

\newtheorem*{quest*}{Question}

\theoremstyle{remark}
\newtheorem{rem}[thm]{Remark}

\newcommand{\defi}[1]{\textsf{#1}}

\begin{document}

\title[Interpolation for space curves]{Interpolation for Brill-Noether space curves\footnote{ L\MakeLowercase{ast} U\MakeLowercase{pdated:} N\MakeLowercase{ovember 26, 2016}}}
\author{Isabel Vogt}
\address{Department of Mathematics, Massachusetts Institute of Technology, Cambridge, MA 02139}
\email{ivogt@mit.edu}

\maketitle

\begin{abstract}
In this note we compute the number of general points through which a general Brill-Noether space curve passes.
\end{abstract}

\section{Introduction}

The goal of this note is to answer the following  fundamental incidence question for space curves:

\begin{question}
Fix a general collection of $n$ points in $\pp^3$ over $\cc$.  When does there exist a curve $C \hookrightarrow \pp^3$ of degree $d$ and genus $g$ and general moduli passing through those points?
\end{question}

In order for there to be a nondegenerate curve $C \hookrightarrow \pp^3$ of degree $d$ and genus $g$ and general moduli, the Brill-Noether number $\rho(d,g,3) = 4d-3g-12$ must be nonnegative.  In that case, there is a unique irreducible component of the Kontsevich space $M_g(\pp^3, d)$ which dominates $M_g$, and whose general member is a nondegenerate immersion of a smooth curve.  We will call this component $M_g(\pp^3, d)^\circ$, and curves in this component \defi{Brill-Noether} curves.

Let $M_{g,n}(\pp^3, d)^\circ$ be the component of $M_{g, n}(\pp^3, d)$ dominating $M_g(\pp^3, d)^\circ$.  Then there is a map 
\[M_{g, n}(\pp^3, d)^\circ \xrightarrow{\ev} \left(\pp^3\right)^n, \]
taking the image of the $n$ marked points.  We are asking in the main question whether this map is dominant.  An obvious prerequisite is that 
\[\dim M_{g,n}(\pp^3, d)^\circ = 4d + n  \geq \dim \left(\pp^3\right)^n = 3n.\]
Hence we would expect to be able to pass a degree $d$ Brill-Noether curve through up to $2d$ general points in $\pp^3$.  The main result of this paper is that this expectation is true with exactly two exceptions:

\begin{thm}\label{main}
There exists a Brill-Noether curve $C$ of degree $d$ and genus $g$ in $\pp^3$ passing through a maximum of
\[\begin{cases} 9 \text{ general points,} & \ \text{ if } (d,g) =(5,2), \text{ or } (6,4) \\
2d \text{ general points,} & \ otherwise. \end{cases}\]
\end{thm}
The fact that curves of $(d,g) = (5,2)$ (resp. $(6,4)$) cannot pass through $10$ (resp. $12$) general points is clear: the curve is forced to lie on a quadric surface in $\pp^3$, which itself can only be passed through $9$ general points.  However a $(2,3)$- (resp. $(3,3)$-) curve on $\pp^1 \times \pp^1$ can be passed through $9$ general points on the quadric since it moves in an $11$ (resp. 15) dimensional family.  These two exceptional cases were already noticed by Atanasov \cite{nasko} and Stevens \cite{stevens} respectively.

Theorem \ref{main} builds upon the work of Perrin \cite{perrin} who investigated the problem via liaison, and Atanasov \cite{nasko}, who answered the question in the nonspecial range ($d \geq g + 3$).  The generalization to Brill-Noether curves in $\pp^r$ (with appropriate choice of $n$) was answered in the nonspecial range in \cite{joint}.  Interpolation problems for higher dimensional varieties were studied in \cite{landesman}, \cite{lp}, and \cite{coskun}.

As in previous work, the current approach is via deformation theory.  First we make the following definition:

\begin{defin} 
We say that a vector bundle $E$ on a smooth curve satisfies the property of \defi{interpolation} if it is nonspecial (e.g., $h^1(E)= 0$) and for every $n\geq 0$, a general effective divisor $D$ of degree $n$ satisfies either
\[ h^0(E(-D)) = 0 \qquad \text{or} \qquad h^1(E(-D)) = 0. \]
\end{defin}

Using this terminology we show the following:

\begin{prop}\label{mainprop}
Let $C$ be a general Brill-Noether curve of degree $d$ and genus $g$ in $\pp^3$.  Then the normal bundle $N_{C/\pp^3}$ satisfies the property of interpolation if and only if $(d,g)$ is not $(5,2)$ or $(6,4)$.
\end{prop}

Let us show that this implies the Theorem \ref{main}.  Denote $(C, p_1, ..., p_n) \in M_{g,n}(\pp^3, d)^\circ$ by the abbreviation $(C, D)$ where $D$ is the Cartier divisor $p_1 + ... + p_n$.  If $h^1(N_C(-D)) = 0$, then the map
\[M_{g, n}(\pp^3, d)^\circ \xrightarrow{\ev} \left(\pp^3\right)^n\] 
is smooth at the point $(C, D \subset C)$.  The Euler characteristic of $N_C$ is $4d$.  Thus if $D$ is a divisor of degree $n$, $\chi(N_C(-D)) = 4d-2n$.
So if $N_C$ satisfies interpolation, then for a general effective divisor $D$ of degree $n$, the following are equivalent:
\begin{enumerate}
\item $h^1(N_C(-D)) = 0$  
\item $ \chi(N_C(-D)) \geq 0 $
\item $ n \leq 2d$
\end{enumerate}
Thus if $n \leq 2d$, the map $\ev$ is smooth at the point $(C, D)$, and hence dominant.

The key input in the present work is the following theorem of Larson:

\begin{thm}[{\cite[Theorem 1.4]{larson}}]\label{quadric}
Let $C$ be a general Brill-Noether curve of degree $d$ and genus $g$ in $\pp^3$.  Then $h^1(N_C(-2)) = 0$ unless
\[(d,g) \in \{ (4,1), (5,2), (6,2), (6,4), (7,5), (8,6)\} .\]
\end{thm}

The relation to the present work is provided by the following lemma:

\begin{lem}\label{genpts}
If $h^1(N_C(-\Gamma)) = 0$ for some effective divisor $\Gamma$ of degree $2d$, then $N_C$ satisfies the property of interpolation. 
\end{lem}
\begin{proof}
Under the assumption $\deg \Gamma = 2d$, we have $\chi(N_C(-\Gamma))=0$.  So $h^1(N_C(-\Gamma)) = 0$ implies that $h^0(N_C(-\Gamma)) = 0$ as well.  If we let $D = E + p$ for some effective divisor $E$ and point $p$, then $h^1(N_C(-E)) \leq h^1(N_C(-D))$ and $h^0(N_C(-D)) \leq h^0(N_C(-E))$.  So by adding and subtracting points from $\Gamma$ we can find effective divisors $D_n$ of every positive degree such that $h^1(N_C(-D_n)) = 0$ if $n \leq 2d$, and $h^0(N_C(-D_n)) = 0$ if $n > 2d$.

Furthermore, if there exists some effective divisor $F$ such that $h^1(N_C(-F)) = 0$ (resp. $h^0(N_C(-F)) = 0$), the semicontinuity theorem implies that $h^1(N_C(-D)) = 0$ (resp. $h^0(N_C(-D)) = 0$) for all $D$ in a nonempty open neighborhood of $F$, which is necessarily dense in $\Sym^{n}C$ as $\Sym^{n}C$ is irreducible.  
\end{proof}

As $\O_C(2)$ is an effective divisor of degree $2d$, all counter-examples to Proposition \ref{mainprop} are necessarily contained in the counterexamples to Theorem \ref{quadric}, so it suffices to verify Proposition \ref{mainprop} in these six cases.  The result of \cite{nasko} in the nonspecial range proves interpolation for $(d,g) = (4,1)$ and $(6,2)$.


Thus to prove Proposition \ref{mainprop} it suffices to consider two cases: curves of degree $7$ and genus $5$, and curves of degree $8$ and genus $6$.  As the property of interpolation is open \cite[Thm 5.8]{nasko}, and the restricted Hilbert scheme parameterizing (limits of) smooth curves of $(d,g) = (7,5)$, respectively $(8,6)$, is irreducible \cite[Thm 2.7]{keem} , it suffices to exhibit one such curve whose normal bundle satisfies interpolation.  Using this, we resolve these two cases in the remaining two sections of the paper.

Although both cases are resolved by realizing $C \subset \pp^3$ as the projection of a canonical curve, the techniques are quite different.  For curves of degree $7$ and genus $5$, we give an explicit description of the sheaf maps arising in a short exact sequence containing $N_{C/\pp^3}$.  This reduces the problem to understanding the generators of the homogenous ideal of a collection of points in $\pp^2$.  For curves of degree $8$ and genus $6$, we find our curve lying on a singular cubic surface.  The ``normal bundle" of $C$ in the singular cubic is then more positive than if $C$ lay on a smooth cubic, which then gives the result.

\begin{rem}
Semistability of normal bundles has previously been studied in \cite{ran, sac} and \cite{ein} for rational curves and elliptic normal curves respectively.
The property of interpolation is quite analogous to that of semistability.  In fact, if $E$ satisfies interpolation and the rank of $E$ divides the Euler characteristic, then $E$ is semistable; indeed, if $E$ satisfies interpolation and $F \subset E$ is a subbundle, then 
\[\frac{\chi(F)}{\rk(F)} \leq \left\lceil \frac{\chi(E)}{\rk(E)} \right\rceil. \]
The converse is not true, even when restricted to bundles which are nonspecial.  To construct an example, let $x$ and $y$ be general points on a curve $C$ of genus $2$.  Then $|K_C(x+y)|$ defines a map $\pi \colon C \to \pp^2$ with image a quartic with a single node.  The vector bundle $ \pi^*(T_{\pp^2})$ has a subbundle of tangent directions pointing towards the node, which is isomorphic to $K_C(2x + 2y)$.  Let $E$ be $\pi^*(T_{\pp^2})(-2x-2y)$.  Then $E$ is rank $2$ and degree $4$.  Further, by the above we have
\[0 \to K_C \to E \to K_C^{\otimes 2}(-x-y) \to 0. \]
As $\deg(K_C) = \deg(K_C^{\otimes 2}(-x-y)) = 2$, $E$ is semistable.  But for any $p \in C$, $h^0(K_C(-p)) = 1$, so $h^0(E(-p)) > 0$.  Thus $E$ does not satisfy interpolation as $\chi(E(-p)) =0 $.  With slightly more work, one can show that $E$ is nonspecial.
\end{rem}

\subsection*{Acknowledgements}  I thank Joe Harris and Bjorn Poonen for numerous helpful discussions and guidance.  I would also like to thank Eric Larson for pointing out that his work in \cite{larson} left this problem within reach, and together with Aaron Landesman, Lawrence Ein, Izzet Coskun and members of the MIT and Harvard mathematics departments for helpful conversations.  Finally, I would like to acknowledge the generous support of the NSF Graduate Research Fellowship Program.

\section{Curves of Degree 7 and Genus 5}

For the remainder of this section, let $C \hookrightarrow \pp^3$ be a curve of degree $7$ and genus $5$ in $\pp^3$.  Let $\Gamma \subset C$ be a general collection of $14$ points on $C$.  We have $\deg(N_C) = 36$ and $\deg(N_C(-\Gamma)) = 8$.  In order to prove that $N_C$ satisfies interpolation, it suffices by Lemma \ref{genpts} to show 
\[ H^1(C, N_C(-\Gamma)) =0.\]

\begin{rem}\label{chi0}
As used in the proof of Lemma \ref{genpts}, when $\deg \Gamma = 2d$, $\chi(N_C(-\Gamma)) = 0$.  So $h^1(N_C(-\Gamma)) = 0 $ if and only if $h^0(N_C(-\Gamma)) = 0$.
\end{rem}


\begin{lem}\label{rrcalc}
Let $L$ be a line bundle of degree $7$ on $C$ with $h^0(L) = 4$.  Then $L = K-p$ for some point $p$ on $C$.
\end{lem}
\begin{proof}
By Riemann-Roch, $h^1(L) = 1$ and hence by Serre duality $h^0(K- L) = 1$.  Thus there exists a unique point $p$ such that $K-L = p$.
\end{proof}

As such, every curve $C$ of degree $5$ and genus $7$ in $\pp^3$ is the projection of a canonical curve $\tilde{C} \hookrightarrow \pp^4$ from a point $ p \in \tilde{C}$.  Call this projection map $\pi \colon \tilde{C} \to C$.  Furthermore, $C \simeq \tilde{C}$ is not trigonal, as projection from $p$ defines an embedding into $\pp^3$.

Let $S \subset \pp^4$ be the cone over $\tilde{C}$ with vertex $p$.  By normal bundle $N_{\tilde{C}/S}$ of $\tilde{C}$ in this (singular) surface $S$, we mean the unique subbundle (i.e. subsheaf with locally free quotient) of $N_{\tilde{C}/\pp^4}$ that agrees with the normal bundle of $C \smallsetminus p$ in $S \smallsetminus p$.  Sections of this bundle have a geometric interpretation as the normal directions that point towards $p$; for this reason, in \cite{joint} $N_{\tilde{C}/S}$ is referred to as $N_{\tilde{C} \to p}$.

\begin{lem}
The sequence
\begin{equation}\label{exact} 0 \to N_{\tilde{C}/S} \to N_{\tilde{C}/\pp^4} \to \pi^* N_{C/\pp^3}(p) \to 0, \end{equation}
is exact.
\end{lem}
\begin{proof}
The projection map $\pi \colon \pp^4 \dashedrightarrow \pp^3$ is resolved by blowing up $\pp^4$ at the point $p$ and as such there is a regular map $\tilde{\pi} \colon \Bl_p\pp^4 \to \pp^3$.  Hence there is a surjective map of sheaves $N_{\tilde{C}/\Bl_p\pp^4} \to N_{C/\pp^3}$.  But $N_{\tilde{C}/\Bl_p \pp^4} \simeq N_{\tilde{C}/\pp^4}(-p)$.  Twisting by $p$ we obtain the right map of \eqref{exact}.  The kernel is the normal directions in the cone $S$. \qedhere
\end{proof}

By \cite[Prop. 6.3]{joint}, we have that $N_{\tilde{C}/S} \simeq \O_{\tilde{C}}(1)(2p)$.  To see this, we need to recall the definition of the Euler field.  Let $V$ be a vector space and $V = V_1 \oplus V_2$ a decomposition with $\dim V_1 = 1$.  Define a $\cc^*$ action on $V$ by
\[\lambda\cdot(x+y) = \lambda x + y, \qquad \lambda \in \cc^*, x \in V_1, y \in V_2. \]
This action descends to $\pp V$, where the decomposition $V = V_1 \oplus V_2$ corresponds to the choice of a point $p$ and a complementary hyperplane $H$.  The \defi{Euler field} associated to $(p, H)$ is the differential of this action at $\lambda = 1 \in \cc^*$.

The Euler field vanishes only along $p$ and $H$.  And further if $p$ is a general point of $\tilde{C}$, $p$ does not lie on any other tangent line.  The restriction of the Euler field to $\tilde{C}$ provides a section of the normal bundle $N_{\tilde{C}/S}$, which vanishes along $\tilde{C} \cap H$ and to some order at $p$.  Explicit calculation gives that this order is $2$.

Furthermore, as $\tilde{C}$ is not trigonal, it is the complete intersection of the net of quadrics on which it lies; hence we have that $N_{\tilde{C}/\pp^4} \simeq \O_{\tilde{C}}(2)^{\oplus 3}$.

The key geometric input to this entire argument is the following lemma, which give a description of the injection in \eqref{exact} in terms of the identifications $N_{\tilde{C}/S} \simeq \O_{\tilde{C}}(1)(2p)$ and $N_{\tilde{C}/\pp^4} \simeq \O_{\tilde{C}}(2)^{\oplus 3}$.

\begin{lem}\label{detmap}
The map
\[ \O_{\tilde{C}}(1)(2p)  \xrightarrow{\alpha} \O_{\tilde{C}}(2)^{\oplus 3} \]
in sequence \eqref{exact} is the unique map (up to isomorphism) given by multiplication by three linearly independent sections of $H^0(K_{\tilde{C}}(-2p))$.
\end{lem}
\begin{proof}
Such a map is specified by three sections of $\underline{\Hom}(\O(1)(2p), \O(2)) = H^0(\O(1)(-2p)) = H^0(K_{\tilde{C}}(-2p))$.  In our case, the map $\alpha$ comes from restricting the Euler field associated to $(p, H)$ to (a choice of) three independent quadrics $Q_1, Q_2, Q_3$ defining $C$.  As a section of $N_{Q_i} \simeq \O_{Q_i}(2)$, the Euler field vanishes along $H$ and $T_pQ_i$, as any point $q$ with $\bar{q,p} \subset T_qQ_i$ is contained in $T_pQ_i$.  The result follows as the three tangent planes $T_pQ_1, T_pQ_2, T_pQ_3$ are linearly independent sections of $H^0(\O_{\tilde{C}}(1)(-2p))$.
\end{proof}

Consider the twist of exact sequence \eqref{exact} by the line bundle $\O_{\tilde{C}}(-p-\Gamma)$:
\begin{equation}\label{twist} 0 \to K_{\tilde{C}}(p-\Gamma) \to K_{\tilde{C}}^{\otimes 2}(-p-\Gamma)^{\oplus 3} \to \pi^*N_{C/\pp^3}(-\Gamma) \to 0. \end{equation}

\begin{lem}
The natural map
\[H^1(\tilde{C}, K_{\tilde{C}}(p-\Gamma)) \to H^1(\tilde{C}, K_{\tilde{C}}^{\otimes 2}(-p-\Gamma)^{\oplus 3}) \]
induced by $\alpha$ is an isomorphism.
\end{lem}
\begin{proof}
By Serre duality, it suffices to show that the map
\[H^0(K_{\tilde{C}}^\vee(p + \Gamma)^{\oplus 3}) \to H^0(\O_{\tilde{C}}(-p + \Gamma)) \]
induced by $\alpha$ is an isomorphism.  As $14 > 5 = \dim \Pic(C)$, the line bundle $L := K_{\tilde{C}}^\vee(p + \Gamma)$ is a general line bundle of degree $-8 + 15 = 7$.  By Lemma \ref{detmap} this map is simply the tensor product map
\[H^0(L) \otimes H^0(K_{\tilde{C}}(-2p)) \to H^0(L \otimes K_{\tilde{C}}(-2p)). \]
The proof therefore follows from the following proposition.
\end{proof}

\begin{prop}\label{tensoriso}
Let $L$ be a general line bundle of degree $7$ and $p$ a general point on a general curve $C$ of genus $5$.  Then the tensor product map
\[H^0(L)  \otimes H^0(K(-2p)) \to H^0(L \otimes K(-2p)), \]
is an isomorphism.
\end{prop}

As $L$ is general, we have $h^0(L) = 3$ and the complete linear series $|L|$ maps $C$ birationally onto a plane septic with generically $10$ nodes.  The linear series $|K(-2p)|$ is therefore cut by quartics passing through the $10$ nodes, and tangent to the septic at $p$.  Let $\Sigma$ denote the zero-dimensional subscheme in $\pp^2$ of deg 12 consisting of the 10 nodes and the tangent vector at $p$.  Let $\I_\Sigma$ be its ideal sheaf.  Then it suffices to show
\begin{equation}\label{idealmap} H^0(\O_{\pp^2}(1)) \otimes H^0(\I_\Sigma(4)) \to H^0(\I_\Sigma(5)) \end{equation}
is an isomorphism.  

We proceed by first showing that the subscheme $\Sigma$ is general among all such degree $12$ subschemes with multiplicity $1$ at $10$ points and multiplicity $2$ at the last point.  Then we will restrict to a canonical curve containing the points and apply the basepoint free pencil trick.

Let $V_{7,5}$ be the Severi variety parameterizing degree $7$ plane curves of geometric genus $5$ and $U_{7,5} \subset V_{7,5}$ the open dense locus of nodal curves with exactly 10 nodes.  There is a map $U_{7,5} \to \Sym^{10}(\pp^2)$ extracting the $10$ nodes.  
Let $H \subset \Hilb^2(\pp^2)$ be locus parameterizing degree $2$ subschemes supported at a single point.  Let $S$ be the incidence correspondence
\[S = \{(C, N, t) \subset U_{7,5} \times \Sym^{10}(\pp^2) \times H \mid N = C_{\text{sing}}, t \subset C \}. \]

\begin{lem}\label{irredu}
The map
\[\pi \colon S \to \Sym^{10}(\pp^2) \times H \]
extracting the nodes and the tangent vector is dominant.
\end{lem}	
\begin{proof}
To show that the map is dominant, we will show that at a general point $(C, N, t)$ of $S$, the map $\pi$ is smooth.  Let $p$ be the support of $t$.  The obstruction to smoothness of $\pi$ at $(C, N, t)$ lies in $H^1(N_f(-f^{-1}(N) - 2p))$, where $N_f$ is the normal sheaf of the normalization map $f \colon C^{\nu} \to \pp^2$ and $f^{-1}(N)$ is the collection of $20$ points on $C^{\nu}$ lying above the nodes of $C$.  As $f$ is unramified, $N_f$ is a line bundle and we have an exact sequence
\[0 \to T_{C^{\nu}} \to f^* T_{\pp^2} \to N_f \to 0. \]
As such, $N_f \simeq K_{C^{\nu}} \otimes f^*(O_{\pp^2}(3))$ is a line bundle of degree $29$ on the genus 5 curve $C^\nu$.  Thus $N_f(-f^{-1}(N))$ has degree $9$ and so by Riemann-Roch
\[h^0(N_f(-f^{-1}(N))) = 5, \qquad h^1(N_f(-f^{-1}(N))) = 0. \]
For simplicity write $F = N_f(-f^{-1}(N))$.  Now we claim that if $p$ is a \emph{general point} of $C^\nu$, then the evaluation map $H^0(F) \to H^0(F|_{2p})$ is surjective.  Indeed, a pencil $[s : t]$ in $|F|$ defines a map to $\pp^1$, which in characteristic $0$ is generically unramified.  So the unique linear combination of $s$ and $t$ which vanishes at a generic point $p$ does not vanish to order $2$.  Using the long exact sequence associated to 
\[ 0 \to F(-2p) \to F \to F|_{2p} \to 0, \]
and the fact already observed that $h^1(F) = 0$, we obtain $h^1(F(-2p)) = h^1(N_f(-f^{-1}(N) - 2p)) = 0$.  Thus the map $\pi$ is smooth at a general point $(C, N, t)$, and hence dominant.
\end{proof}

\begin{lem}\label{quartic}
Let $X$ be a smooth degree $4$ curve in $\pp^2$, and let $D \subset X$ be a general effective divisor of degree $12$ supported at $11$ general points.  Then the map
\[ H^0(K_X) \otimes H^0(4K_X - D) \to H^0(5K_X - D), \]
is an isomorphism.
\end{lem}

\begin{proof}
The line bundle $\O_X(D)$ is a general line bundle of degree $12$ on $X$.  As such, $H^0(4K_X - D) = 2$ and for any point $p \in X$, $H^0(4K_X - D - p) = 1$ (e.g. it is basepoint-free).  By the basepoint-free pencil trick \cite[III.3]{acgh}, the kernel of the above tensor-product map is $H^0(3K_X - D) = 0$.
\end{proof}

\begin{proof}[Proof of Proposition \ref{tensoriso}]
We will show that \eqref{idealmap} is an isomorphism.  By Lemma \ref{irredu}, the subscheme $\Sigma$ can be choose to be a general point in $\Sym^{10}(\pp^2) \times H$.  Since the isomorphism in \eqref{idealmap} is an open condition on subschemes $\Sigma$, we can specialize these points to lie on a smooth quartic $X$.  As both sides of \eqref{idealmap} are dimension $9$, and the kernel of restriction $H^0(\I_{\Sigma}(5)) \to H^0(5K_X - D)$, which is generated by the equation of $X$, is clearly in the image of the tensor product map, it suffices to show the isomorphism when restricted to $X$.  The result now follows from Lemma \ref{quartic}.
\end{proof}

In the long exact sequence associated to \eqref{twist}, the map
\[H^1(\tilde{C}, K_{\tilde{C}}(p-\Gamma)) \to H^1(\tilde{C}, K_{\tilde{C}}^{\otimes 2}(-p-\Gamma)^{\oplus 3}) \]
is an isomorphism, and hence $H^1(\tilde{C}, \pi^*N_{C/\pp^3}(-\Gamma)) = 0$.  This completes the proof of interpolation for curves of degree $7$ and genus $5$.

\section{Curves of Degree 8 and Genus 6}

In this section, $C \hookrightarrow \pp^3$ will denote a curve of degree $8$ and genus $6$.  Let $\Gamma$ be a general set of $16$ points on $C$.  In this case $\deg N_C = 42$ and $\deg N_C(-\Gamma) = 10$.  We want to show that for a general Brill-Noether curve $C$, we have $h^1(N_C(-\Gamma)) = 0$.  Equivalently by Remark \ref{chi0}, we need to show that $h^0(N_C(-\Gamma)) = 0$.

As in Lemma \ref{rrcalc}, $\O_C(1) = K_C - p-q$ for two points $p, q \in C$.  Thus the curve $C \hookrightarrow \pp^3$ is the projection of the canonical curve $\tilde{C} \subset \pp^5$ from the secant line $\ell_{p,q} = \bar{p,q} \subset \pp^5$.  
A general canonical curve $\tilde{C}$ in $\pp^5$ is a quadric section of the unique del Pezzo surface $T$ of degree $5$ on which it lies.  As such $N_{\tilde{C}/T} = \O_{\tilde{C}}(2) = 2K_C$ is a line bundle of degree $20$.

The surface $T$ is the blowup of $\pp^2$ at $4$ points $p_1, p_2, p_3, p_4$ in linear general position, embedded by the anticanonical series.  Let $p = \bar{p_1, p_2} \cap \bar{p_3, p_4}$ in $\pp^2$.  By slight abuse of notation we will also write $p$ for the preimage in $T$.  Let $q$ be a generic point on $\bar{p_1, p_3}$.

Let $\hat{T} \colonequals \Bl_{p,q} T \xrightarrow{\beta} T$.  Denote by $E_p$ (resp. $E_q$) the exceptional divisor $\beta^{-1}(p)$ (resp. $\beta^{-1}(q)$).  Note that $\hat{T}$ is also the blowup of $\pp^2$ in the $6$ (non-generic) points $\{p_1, p_2, p_3, p_4, p, q\}$.  Let $L_{ij}$ be the proper transform in $\hat{T}$ of the line $\bar{p_i, p_j}$ joining $p_i$ and $p_j$ in $\pp^2$.  

\begin{lem}\label{cubics}
The anticanonical linear system $|-K_{\hat{T}}| = \left|\beta^*(\O_T(1))(-E_p-E_q)\right|$ is basepoint-free.  Furthermore, for any zero dimensional, degree two subscheme $D \subset \hat{T}$ with $D \not\subset L_{12}$ and $D \not\subset L_{23}$ and $D \not\subset L_{13}$, let $\I_D$ denote the ideal sheaf of $D$.  Then we have
\[h^0(\hat{T}, -K_{\hat{T}}\otimes \I_D) = h^0(\hat{T}, -K_{\hat{T}}) - 2. \]
\end{lem}
\begin{proof}
This well-known result is claimed without proof on page 643 in \cite{gh}.  It also follows from Th\'eor\`eme 1 in \cite[IV.3]{demazure}.
\end{proof}

We can arrange for a canonical curve $\tilde{C} \subset T$ to pass through the points $p$ and $q$, as it is a quadric section of $T$.  Let $\ell_{p,q} \subset \pp^5$ be the line in $\pp^5$ joining $p$ and $q$ (not to be confused with the proper transform of the line joining $p$ and $q$ in $\pp^2$!)  With such a choice, the line $\ell_{p,q}$ meets $T$ only at $p$ and $q$;  indeed, the line $\ell_{p,q}$ is the intersection of all hyperplanes containing $p$ and $q$.  Since $\beta^*(\O_T(1))(-E_p-E_q)$ is basepoint-free, all hyperplanes containing $p$ and $q$ cannot meet $T$ in another point.

Let $\pi \colon T \dashedrightarrow \pp^3$ be projection from $\ell_{p,q}$.  This is resolved by blowing up $p$ and $q$; we will write $\hat{T} = \Bl_{p,q} T \xrightarrow{\beta} T$.  Then write $\hat{\pi} \colon \hat{T} \to \pp^3$ for the resolved projection map.
Let $\hat{C}$ be the proper transform of $\tilde{C}$ in $\hat{T}= \Bl_{p,q}T$.  The class of $C$ on $\hat{T}$ is thus $\beta^*(\O_T(2))(- E_p - E_q) = -2K_{\hat{T}}(E_p+E_q)$.  Because we will need it latter, we will show now:

\begin{lem}\label{bpf}
The class $[\hat{C}] = -2K_{\hat{T}}(E_p+E_q)$ on $\hat{T}$ is basepoint-free.
\end{lem}
\begin{proof}
As $-K_{\hat{T}}$ is basepoint free, $-K_{\hat{T}} + E_p$ could only have basepoints along $E_p$.  Similarly the only basepoints of $-K_{\hat{T}} +E_q$ could occur along $E_q$.  Hence the sum is basepoint free.
\end{proof}

Note that we have
\[ [\hat{C}] \cdot L_{12} = [\hat{C}] \cdot L_{34} = [\hat{C}] \cdot L_{13} = 1.\]
With this special choice of $p$ and $q$, Lemma \ref{cubics} guarantees that the map $\hat{\pi}$ still gives an embedding of $\hat{C}$ into $\pp^3$, since $\hat{C}$ meets each line of $L_{12} \cup L_{13} \cup L_{34}$ only once. 

As indicated in the introduction, the idea here is to exhibit a sub line bundle $L$ of $N_{C/\pp^3}$ of such a degree that it forces $h^0(N_{C}(-\Gamma)) = 0$.  Indeed, by Riemann-Roch, a general line bundle of degree $5$ on a genus $6$ curve has no global sections.  If $L \hookrightarrow N_C$ has degree $21$, the quotient $Q$ has degree $21$ as well.  After twisting down by $16$ general points, $L(-\Gamma)$ and $Q(-\Gamma)$ will both be general bundles of degree $5$, and hence force $h^0(N_C(-\Gamma) ) = 0$.

One way of producing subbundles of $N_{C/\pp^3}$ is to exhibit your curve on certain surfaces $S$: then $N_{C/S} \hookrightarrow N_{C/\pp^3}$.  A general curve of degree $8$ and genus $6$ lies on no planes or quadrics, but it does lie on a unique cubic surface.  Unfortunately, the degree of the normal bundle in a smooth cubic is $18$.  Our idea here is to use a \emph{singular} cubic surface, which will be the image of $\hat{T}$ under the map $\hat{\pi}$: a cubic surface with three ordinary double points.  The key lemma, which allows us to relate this to the normal bundle in the desingularization, is the following:

\begin{lem}\label{subdown}
Let $C \subset X$ be an embedding of a smooth curve in a surface and $\pi \colon X \to Y$ a generically unramified map of smooth varieties, whose composition with $C \hookrightarrow X$ gives an embedding $C \hookrightarrow Y$.  Let $E \subset X$ be the subscheme where $d\pi \colon T_X \to \pi^* T_Y$ drops rank.  Then
\[ N_{C/X}(E \cap C) \hookrightarrow \pi^*N_{C/Y}, \]
is an injection of vector bundles.
\end{lem}
\begin{proof}
The map $d\pi \colon T_X|_C/T_C \to T_Y|_C/T_C $ is an injection of vector bundles away from $E\cap C$.  At points $p$ of $E \cap C$, the differential drops rank, but as $\pi$ restricts to an immersion along $C$, the differential must vanish along a subspace of $T_{X, p}$ complementary to $T_{C,p}$.  Hence it vanishes on the fiber $T_{X,p}/T_{C,p}$ of $N_{C/X}$ at $p$ exactly to the order of $p$ in $E \cap C$.  All together, $N_{C/X} \hookrightarrow \pi^*N_{C/Y}(-E)$ as desired.
\end{proof}

We will apply this lemma both to the map $\beta \colon \Bl_{p,q}T \to T$ and $\hat{\pi} \colon \Bl_{p,q}T \to \pp^3$.  
The scheme on which $d \beta$ drops rank is the effective divisor $\Ram(\beta) = E_p + E_q \in H^0(K_{\hat{T}} - \beta^{*}K_{T})$ on which $\det(d\beta)$ vanishes \cite[1.41]{debarre}.
Hence $N_{C/\hat{T}}(p+q) \hookrightarrow \beta^*N_{C/T}$ is an \emph{inclusion of line bundles} and thus an isomorphism.  As such, $N_{C/\hat{T}} \simeq \beta^*N_{C/T}(-p-q)$ is a line bundle of degree $20 - 2 = 18$.

By Lemma \ref{cubics}, the map $\hat{\pi} \colon \hat{T} \to \pp^3$ is an embedding away from the lines $ L_{12} \cup L_{34} \cup L_{13}$, and hence the scheme $E\subset \hat{T}$ where $d\hat{\pi}$ drops rank is supported on $ L_{12} \cup L_{34} \cup L_{13}$.  We claim that it is scheme-theoretically equal to this locus.  Indeed, the map $\hat{\pi}$ is just the complete linear system of cubics in $\pp^2$ through the six points $p_1, p_2, p_3, p_4, p,q$.  The curve $L_{12}$ is contracted by $\hat{\pi}$ because it is only a codimension $1$ condition for such a cubic to contain $L_{12}$: it simply must contain one point on the line.  If a higher multiple of the line $ L_{12}$ is contracted by $\hat{\pi}$, then the codimension of the space cubics containing the double line $L_{12}$ is less than or equal to $2$.  But there are no such cubics containing $2L_{12}$ \emph{and} passing through $p_1, p_2, p_3, p_4, p,q$.  Similarly for $L_{34}$ and $L_{13}$.  Thus we conclude that $E = L_{12} \cup L_{34} \cup L_{13}$ away from codimension $2$.  We can ignore codimension $2$ phenonmena, as the class of $[C] = -2K_{\hat{T}} + E_p +E_q$ is basepoint-free by Lemma \ref{bpf}.

Putting this all together, 
\[N_{C/\hat{T}}(E \cap C) \hookrightarrow \hat{\pi}^*N_{C/\pp^3}, \]
is an inclusion of vector bundles.  Thus $N_{C/\pp^3}$ has a subbundle of rank $1$ and degree $21$, as desired.  Using the above observation, for such a curve $C$, $h^1(N_C(-\Gamma)) = 0$, which completes the proof of interpolation.

\bibliographystyle{plain}
\bibliography{Interpolation}

\begin{thebibliography}{10}

\bibitem{acgh}
E.~Arbarello, M.~Cornalba, P.~A. Griffiths, and J.~Harris.
\newblock {\em Geometry of algebraic curves. {V}ol. {I}}, volume 267 of {\em
  Grundlehren der Mathematischen Wissenschaften [Fundamental Principles of
  Mathematical Sciences]}.
\newblock Springer-Verlag, New York, 1985.

\bibitem{joint}
Atanas Atanasov, Eric Larson, and David Yang.
\newblock Interpolation for normal bundles of general curves.
\newblock Available at \url{http://arxiv.org/abs/1509.01724v3}.

\bibitem{nasko}
Atanas~Valeryev Atanasov.
\newblock {\em Interpolation and vector bundles on curves}.
\newblock ProQuest LLC, Ann Arbor, MI, 2015.
\newblock Thesis (Ph.D.)--Harvard University.

\bibitem{coskun}
Izzet Coskun.
\newblock Degenerations of surface scrolls and the {G}romov-{W}itten invariants
  of {G}rassmannians.
\newblock {\em J. Algebraic Geom.}, 15(2):223--284, 2006.

\bibitem{debarre}
Olivier Debarre.
\newblock {\em Higher-dimensional algebraic geometry}.
\newblock Universitext. Springer-Verlag, New York, 2001.

\bibitem{demazure}
Michel Demazure, Henry~Charles Pinkham, and Bernard Teissier, editors.
\newblock {\em S\'eminaire sur les {S}ingularit\'es des {S}urfaces}, volume 777
  of {\em Lecture Notes in Mathematics}.
\newblock Springer, Berlin, 1980.
\newblock Held at the Centre de Math{\'e}matiques de l'{\'E}cole Polytechnique,
  Palaiseau, 1976--1977.

\bibitem{ein}
Lawrence Ein and Robert Lazarsfeld.
\newblock Stability and restrictions of {P}icard bundles, with an application
  to the normal bundles of elliptic curves.
\newblock In {\em Complex projective geometry ({T}rieste, 1989/{B}ergen,
  1989)}, volume 179 of {\em London Math. Soc. Lecture Note Ser.}, pages
  149--156. Cambridge Univ. Press, Cambridge, 1992.

\bibitem{gh}
Phillip Griffiths and Joseph Harris.
\newblock {\em Principles of algebraic geometry}.
\newblock Wiley Classics Library. John Wiley \& Sons, Inc., New York, 1994.
\newblock Reprint of the 1978 original.

\bibitem{keem}
C.~Keem and SeonJa Kim.
\newblock Irreducibility of a subscheme of the {H}ilbert scheme of complex
  space curves.
\newblock {\em J. Algebra}, 145(1):240--248, 1992.

\bibitem{landesman}
Aaron Landesman.
\newblock Interpolation of varieties of minimal degree.
\newblock Available at \url{http://arxiv.org/abs/1605.01492v1}.

\bibitem{lp}
Aaron Landesman and Anand Patel.
\newblock Interpolation problems: Del pezzo surfaces.
\newblock Available at \url{http://arxiv.org/abs/1601.05840v1}.

\bibitem{larson}
Eric Larson.
\newblock The generality of a section of a curve.
\newblock Available at \url{http://arxiv.org/abs/1605.06185v2}.

\bibitem{perrin}
Daniel Perrin.
\newblock Courbes passant par {$k$} points g\'en\'eraux de {$\pp^3$}.
\newblock {\em C. R. Acad. Sci. Paris S\'er. I Math.}, 299(10):451--453, 1984.

\bibitem{ran}
Ziv Ran.
\newblock Normal bundles of rational curves in projective spaces.
\newblock {\em Asian J. Math.}, 11(4):567--608, 2007.

\bibitem{sac}
Gianni Sacchiero.
\newblock Normal bundles of rational curves in projective space.
\newblock {\em Ann. Univ. Ferrara Sez. VII (N.S.)}, 26:33--40 (1981), 1980.

\bibitem{stevens}
Jan Stevens.
\newblock On the number of points determining a canonical curve.
\newblock {\em Nederl. Akad. Wetensch. Indag. Math.}, 51(4):485--494, 1989.

\end{thebibliography}

\end{document}